\documentclass[12]{amsart}
\usepackage{amsmath,amssymb,amsthm,color,enumerate,comment,centernot,enumitem,url,cite}
\usepackage{graphicx,relsize,bm}
\usepackage{mathtools}
\usepackage{array}

\makeatletter
\newcommand{\tpmod}[1]{{\@displayfalse\pmod{#1}}}
\makeatother

\newtheorem{thm}{Theorem}[section]
\newtheorem{lemma}[thm]{Lemma}

\theoremstyle{remark}

\theoremstyle{definition}

\theoremstyle{THM}

\newcommand{\abs}[1]{\left|{#1}\right|}

\def\FF {{\mathcal F}}

\def\Z {{\mathbb Z}}

\def\Q {{\mathbb Q}}

\def\F {{\mathbb F}}

\def\Z {{\mathbb Z}}
\def\Q {{\mathbb Q}}

\def\Gal{{\mbox{{\rm{Gal}}}}}

\makeatletter
\@namedef{subjclassname@2020}{%
  \textup{2020} Mathematics Subject Classification}
\makeatother

\def\red#1 {\textcolor{red}{#1 }}
\def\blue#1 {\textcolor{blue}{#1 }}

\numberwithin{equation}{section}

\def\Z {{\mathbb Z}}

\newcommand{\Mod}[1]{\ (\mathrm{mod}\enspace #1)}
\newcommand{\mmod}[1]{\ \mathrm{mod}\enspace #1}
\begin{document}

\title[Monogenic Reciprocal Quartic Polynomials]{Monogenic Reciprocal Quartic Polynomials And Their Galois Groups}


\author{Lenny Jones}
\address{Professor Emeritus, Department of Mathematics, Shippensburg University, Shippensburg, Pennsylvania 17257, USA}
\email[Lenny~Jones]{doctorlennyjones@gmail.com}

\date{\today}

\begin{abstract}
Suppose that $f(x)=x^4+Ax^3+Bx^2+Ax+1\in \Z[x]$. We say that $f(x)$ is {\em monogenic} if $f(x)$ is irreducible over $\Q$ and $\{1,\theta,\theta^2,\theta^3\}$ is a basis for the ring of integers of $\Q(\theta)$, where $f(\theta)=0$.
For each possible Galois group $G$ that can occur in the two cases of $A\ne 0$ with $B=0$, and $AB\ne 0$, we determine all
  monogenic polynomials $f(x)$ with Galois group $G$.
\end{abstract}


\maketitle
\section{Introduction}\label{Section:Intro}
Our focus in this article is on the monogenicity and Galois groups over $\Q$ of reciprocal quartic polynomials
\begin{equation}\label{Eq:f}
f(x):=x^4+Ax^3+Bx^2+Ax+1\in \Z[x],
\end{equation} in two particular cases: $A\ne 0$ with $B=0$, and $AB\ne 0$. 
We remind the reader that $f(x)$ is {\em monogenic} if $f(x)$ is irreducible over $\Q$ and $\{1,\theta,\theta^2,\theta^3\}$ is a basis for the ring of integers of the quartic field $\Q(\theta)$, where $f(\theta)=0$. We say that two irreducible quartic polynomials are {\em equivalent} if they generate the same quartic field. Otherwise, we say they are {\em distinct}.

Our approach involves three stages. First, we derive necessary and sufficient conditions on $A$ and $B$ so that $f(x)$ is irreducible over $\Q$. Second, we use a Theorem of Dedekind (see Theorem \ref{Thm:Dedekind}) to impose restrictions on $A$ and $B$ so that $f(x)$ is monogenic. In the third and final stage, for irreducible $f(x)$ and each possible group $G$, we layer the additional restrictions on $A$ and $B$ such that the Galois group of $f(x)$ over $\Q$, denoted $\Gal(f)$, is $G$. Ultimately, this three-stage procedure allows us to determine all monogenic polynomials $f(x)$ with $\Gal(f)\simeq G$. Moreover, in the first case of $A\ne 0$ with $B=0$, we determine all distinct monogenic polynomials. In the second case of $AB\ne 0$, we show that the monogenic polynomials can be partitioned into five sets, four of which are infinite. Each of the four infinite sets contains an infinite subset of distinct polynomials, while the fifth set is a six-element set having a maximal subset of three distinct polynomials. Furthermore, $\Gal(f)\simeq D_4$ if $f(x)$ is an element of any of the infinite sets, and $\Gal(f)\simeq C_4$ for all polynomials $f(x)$ in the six-element set. 

Throughout this article, we let $f(x)$ be as defined in \eqref{Eq:f}. We let $C_n$ denote the cyclic group of order $n$, and $D_4$ denote the dihedral group of order 8. We also let
\[W_1:=B+2-2A,\quad W_2:=B+2+2A, \quad W_3:=A^2-4B+8\]
\[\mbox{and}\quad W:=W_1W_2W_3.\]

Our main theorem is:
\begin{thm}
\text{} 
 \begin{enumerate}
  \item \label{M:I1} Suppose that $A\not \in \{0,\pm 1\}$ and $B=0$. 
  \begin{enumerate}
  \item \label{M:I10}  Then $f(x)$ is monogenic if and only if $2\mid A$ and
  \[A-1, \quad A+1 \quad \mbox{and} \quad (A/2)^2+2\]
  are all squarefree;
    \item \label{M:I11} The set $\FF$ of the monogenic polynomials $f(x)$ is infinite;
    \item \label{M:I12} $\Gal(f)\simeq D_4$ for each $f(x)\in \FF$;
    \item \label{M:I13} The set $\FF^{+}:=\{f(x)\in \FF: A\ge 2\}$ is a complete set of distinct monogenic quartic polynomials of this form.
  \end{enumerate}
  \item \label{M:I2} Suppose that $AB\ne 0$ and $W_3$ is not a square.
    \begin{enumerate}[label=(\Roman*)]
  \item \label{M:I21} Then $f(x)$ is monogenic if and only if all of the following hold: 
  \begin{enumerate}[label=(\alph*)]
    \item  $W_1$ is squarefree;
    \item  $W_2$ is squarefree;
    \item  if $2\mid A$ and $2\mid B$, then $W_3/4$ is squarefree and
    \[(A\mmod{4}, \ B\mmod{4})\in\{(0,0),(2,0)\};\]
    \item  if $2\mid A$ and $2\nmid B$, then $W_3/4$ is squarefree and
    \[(A\mmod{4}, \ B\mmod{4})\in\{(0,3),(2,1)\};\]
    \item  if $2\nmid A$ and $2\mid B$, then $W_3$ is squarefree and
    \[(A\mmod{4}, \ B\mmod{4})\in\{(1,2),(3,2)\};\]
    \item  if $2\nmid AB$, then $W_3$ is squarefree and
    \[(A\mmod{4}, \ B\mmod{4})\in\{(1,1),(1,3),(3,1),(3,3)\}.\]
    \end{enumerate}
    \item \label{M:I22}
     Suppose that $f(x)$ is monogenic, and let
    \begin{enumerate}[label=(\roman*)]
      \item[] $\FF_1:=\{f(x): 2\mid A \ \mbox{and}\ 2\mid B\}$;
      \item[] $\FF_2:=\{f(x): 2\mid A \ \mbox{and}\ 2\nmid B\}$;
      \item[] $\FF_3:=\{f(x): 2\nmid A \ \mbox{and}\ 2\mid B\}$;
      \item[] $\FF_4:=\{f(x): 2\nmid AB \ \mbox{and $W$ is not a square}\}$;
      \item[] $\FF_5:=\{f(x): 2\nmid AB \ \mbox{and $W$ is a square}\}$.
    \end{enumerate}
    Then the sets $\FF_i$ partition the set of all monogenic polynomials with $AB\ne 0$. Moreover, for each $i\in \{1,2,3,4\}$, $\Gal(f)\simeq D_4$ for every $f(x)\in \FF_i$, and each $\FF_i$ contains infinitely many distinct monogenic polynomials. The set $\FF_5$ contains precisely the six monogenic polynomials
    \begin{equation}\label{C4polys}
        \begin{array}{c}
     x^4+\varepsilon x^3+x^2+\varepsilon x+1, \quad x^4+\varepsilon 9x^3+19x^2+\varepsilon 9x+1,\\[.5em]
      x^4+\varepsilon 11x^3+31x^2+\varepsilon 11x+1,
     \end{array}\end{equation} where $\varepsilon \in \{\pm 1\}$, and $\Gal(f)\simeq C_4$ for each $f(x)\in \FF_5$. Furthermore, the three-element set of the polynomials with $\varepsilon=1$ in \eqref{C4polys} is a maximal subset of $\FF_5$ consisting of distinct polynomials.
\end{enumerate}
\end{enumerate}
\end{thm}

\section{Preliminaries}
We begin with some information surrounding the monogenicity of $f(x)$. 
Suppose that $f(x)$ is irreducible over $\Q$, and
let $K=\Q(\theta)$ with ring of integers $\Z_K$, where $f(\theta)=0$. Then, we have \cite{Cohen}
\begin{equation} \label{Eq:Dis-Dis}
\Delta(f)=\left[\Z_K:\Z[\theta]\right]^2\Delta(K),
\end{equation}
where $\Delta(f)$ and $\Delta(K)$ denote the discriminants over $\Q$, respectively, of $f(x)$ and the number field $K$.
Thus, we have the following theorem from \eqref{Eq:Dis-Dis}.
\begin{thm}\label{Thm:mono}
Suppose that $f(x)$ is irreducible over $\Q$. Then $f(x)$ is monogenic if and only if
  $\Delta(f)=\Delta(K)$, or equivalently, $\Z_K=\Z[\theta]$.
\end{thm}
\noindent
 An easy calculation in Maple shows that
\begin{equation}\label{Eq:Delf}
  \Delta(f)=W_1W_2W_3^2. 
\end{equation}

The next theorem deals with the irreducibility of $f(x)$. 
\begin{thm}\label{Thm:Main1}\text{} 
  \begin{enumerate}
    \item \label{M1:I1} If $A=0$ and $B=0$, then $f(x)=x^4+1$ is irreducible.
    \item \label{M1:I2} If $A=0$ and $B\ne 0$, then $f(x)$ is reducible if and only if
    \[\mbox{at least one of} \quad -B-2, \ -B+2 \quad \mbox{and} \quad B^2-4 \quad \mbox{is a square.}\]
    \item \label{M1:I3} If $A\ne 0$ and $B=0$, then $f(x)$ is reducible if and only if $A=\pm 1$.
    \item \label{M1:I4} If $AB\ne 0$, then $f(x)$ is reducible if and only if $W_3$ is a square.
  \end{enumerate}
\end{thm}
\begin{proof}
Item \eqref{M1:I1} is obvious.
Item \eqref{M1:I2} follows from \cite[Theorem 2]{KW}.

For item \eqref{M1:I3}, we have $f(x)=x^4+Ax^3+Ax+1$, where $A\ne 0$. Suppose first that $f(x)$ is reducible. If $f(x)$ has a linear factor, then $f(1)=2(A+1)=0$ or $f(-1)=-2(A-1)=0$, by the Rational Root theorem. Hence, $A=\pm 1$. If $f(x)$ has a quadratic factor, then either
\begin{equation}\label{Eq:Quadfac1}
f(x)=(x^2+ax+1)(x^2+bx+1)=x^4+(a+b)x^3+(ab+2)x^2+(a+b)x+1
\end{equation}
or
\begin{equation}\label{Eq:Quadfac2}
f(x)=(x^2+ax-1)(x^2+bx-1)=x^4+(a+b)x^3+(ab-2)x^2-(a+b)x+1,
\end{equation}
for some $a,b\in \Z$. Since \eqref{Eq:Quadfac2}
implies that $A=0$, we must have \eqref{Eq:Quadfac1}.
 Hence, $ab+2=0$, so that $(a,b)\in \{(-2,1),(-1,2),(1,-2),(2,-1)\}$, which implies that $A=\pm 1$, and the proof is complete in this direction. The converse follows easily by checking that when $A=\pm 1$, we get a nontrivial factorization of $f(x)$.

Finally, for item \eqref{M1:I4}, suppose first that $f(x)$ is reducible. If $f(x)$ has a linear factor, then we get from the Rational Root theorem that
\[\mbox{either} \quad f(1)=W_2=0 \quad \mbox{or} \quad f(-1)=W_1=0.\] If $f(x)$ has a quadratic factor, then we must have \eqref{Eq:Quadfac1}, since \eqref{Eq:Quadfac2} implies that $A=0$. Thus, equating coefficients, we get the system of equations
\[A=a+b \quad \mbox{and} \quad  B=ab+2.\] Then, $b=A-a$, so that $B=a(A-a)+2$, which gives the quadratic equation
\[a^2-Aa+B-2=0\]
in $a$. Solving this equation for $a$ yields
\begin{equation}\label{Eq:ab}
a=\frac{A\pm \sqrt{W_3}}{2} \quad \mbox{and}\quad b=\frac{A\mp\sqrt{W_3}}{2}.
\end{equation}  It follows that $W_3$ must be a square. Note then that, regardless of the parity of $A$, we have $a,b\in \Z$, which completes the proof in this direction. Conversely, an easy computation confirms that if $W_3$ is a square, then with $a$ and $b$ as given in \eqref{Eq:ab}, we get a nontrivial factorization of $f(x)$.
 \end{proof}

Although Theorem \ref{Thm:Main1} can be found spread out in the literature \cite{KW,Dickson}, the comprehensive organization and the elementary methods used here for the proof of items \eqref{M1:I3} and \eqref{M1:I4} yield a consolidated, more user-friendly version. Concerning the Galois groups and monogenicity of $f(x)$, the case in item \eqref{M1:I2} of Theorem \ref{Thm:Main1} has  been previously addressed in \cite{JonesBAMSEQT} and \cite{HJBAMS}, where it was shown, respectively, that there are no reciprocal monogenic trinomials with Galois group $C_4$, while there are infinitely many reciprocal monogenic trinomials with Galois group $C_2\times C_2$. The quintinomials of item \eqref{M1:I4} have been partially investigated in \cite{JonesBAMSQuints} and \cite{JonesFACM}. In \cite{JonesBAMSQuints}, the goal was to determine when quintinomials of degree $2^n$ are monogenic for all $n\ge 2$, without regard to their Galois groups. Although the special case of $A\equiv B\equiv 1\pmod{4}$ was omitted in \cite{JonesBAMSQuints}, it was later addressed in \cite{JonesFACM}, along with their Galois groups in certain cases.
As indicated in Section \ref{Section:Intro}, our focus here is on items \eqref{M1:I3} and \eqref{M1:I4} of Theorem \ref{Thm:Main1}. In particular, we wish to treat the situations not previously addressed in the literature.

The following lemma, which is due to Awtrey and Patane \cite[Corollary 4.2]{AP}, addresses the Galois group of $f(x)$. 
\begin{lemma} \label{Lem:AP} Let $f(x)$ be irreducible over $\Q$. 
 Then $Gal(f)$ is isomorphic to
\begin{enumerate}
  \item $C_4$ if and only if $W$ is a square;
      \item $C_2\times C_2$ if and only if $W_1W_2$ is a square;
      \item $D_4$ if and only if neither $W_1W_2$ nor $W$ is a square.
  \end{enumerate}
\end{lemma}

The following theorem, known as \emph{Dedekind's Index Criterion}, or simply \emph{Dedekind's Criterion} if the context is clear, is a standard tool used in determining the monogenicity of a polynomial.
\begin{thm}[Dedekind \cite{Cohen}]\label{Thm:Dedekind}
Let $K=\Q(\theta)$ be a number field, $T(x)\in \Z[x]$ the monic minimal polynomial of $\theta$, and $\Z_K$ the ring of integers of $K$. Let $p$ be a prime number and let $\overline{ * }$ denote reduction of $*$ modulo $p$ (in $\Z$, $\Z[x]$ or $\Z[\theta]$). Let
\[\overline{T}(x)=\prod_{i=1}^k\overline{\tau_i}(x)^{e_i}\]
be the factorization of $T(x)$ modulo $p$ in $\F_p[x]$, and set
\[h_1(x)=\prod_{i=1}^k\tau_i(x),\]
where the $\tau_i(x)\in \Z[x]$ are arbitrary monic lifts of the $\overline{\tau_i}(x)$. Let $h_2(x)\in \Z[x]$ be a monic lift of $\overline{T}(x)/\overline{h_1}(x)$ and set
\[F(x)=\dfrac{h_1(x)h_2(x)-T(x)}{p}\in \Z[x].\]
Then
\[\left[\Z_K:\Z[\theta]\right]\not \equiv 0 \pmod{p} \Longleftrightarrow \gcd\left(\overline{F},\overline{h_1},\overline{h_2}\right)=1 \mbox{ in } \F_p[x].\]
\end{thm}

\section{The Proof of Theorem \ref{Thm:Main1}}
\begin{proof} We present the proof in sections for the convenience of the reader.
\subsection*{The proof of item \eqref{M:I10}}\hfill\\
      Note that $f(x)$ is irreducible over $\Q$ by  item \eqref{M1:I3} of Theorem \ref{Thm:Main1} since $A\not \in \{0,\pm 1\}$ and $B=0$. Observe that
  \begin{equation*}\label{Eq:Delf1}
  \Delta(f)=-4(A-1)(A+1)(A^2+8)^2
  \end{equation*} from \eqref{Eq:Delf}. For prime divisors $p$ of $\Delta(f)$, we use Theorem \ref{Thm:Dedekind} with $T(x):=f(x)$ to determine necessary and sufficient conditions for the monogenicity of $T(x)$. Let $K=\Q(\theta)$ with ring of integers $\Z_K$, where $T(\theta)=0$.

\subsection*{The special case of $p=2$}\hfill\\ 
Suppose first that $p=2$. Then
  \[\overline{T}(x)=\left\{
  \begin{array}{cl}
  (x+1)^4 & \mbox{if $2\mid A$}\\[.5em]
  (x+1)^2(x^2+x+1)& \mbox{if $2\nmid A$.}
  \end{array}\right.\] Thus, in Theorem \ref{Thm:Dedekind}, we can let
  \[\begin{array}{cl}
  h_1(x)=x+1 \quad \mbox{and} \quad h_2(x)=(x+1)^3 & \quad \mbox{if $2\mid A$}\\[.5em]
  h_1(x)=(x+1)(x^2+x+1) \quad \mbox{and} \quad h_2(x)=(x+1) & \quad \mbox{if $2\nmid A$.}
  \end{array}\]  Hence, in either case, we simply have to calculate $\overline{F}(-1)$ to determine whether or not
  $\gcd(\overline{F},\overline{h_1},\overline{h_2})=1$. Since
  \[F(x)=\frac{h_1(x)h_2(x)-T(x)}{2}=\left\{\begin{array}{cl}
  \left(\dfrac{-A+4}{2}\right)x^3+3x^2+\left(\dfrac{-A+4}{2}\right)x & \mbox{if $2\mid A$}\\[1em]
  \left(\dfrac{-A+3}{2}\right)x^3+2x^2+\left(\dfrac{-A+3}{2}\right)x & \mbox{if $2\nmid A$,}
  \end{array}\right.\] it follows that $F(-1)=A-1$ in both cases. Thus,
  \[\overline{F}(-1)=\left\{\begin{array}{cl}
    1 & \mbox{if $2\mid A$}\\[.5em]
    0 & \mbox{if $2\nmid A$,}
  \end{array}\right.\] which implies that
   \[\gcd(\overline{F},\overline{h_1},\overline{h_2})=1 \quad \mbox{if and only if} \quad 2\mid A.\] Consequently,
   \[2\nmid [\Z_K:\Z[\theta]] \quad \mbox{if and only if} \quad 2\mid A.\] Therefore, we assume that $2\mid A$ and proceed.
\subsection*{The case $p\mid (A-1)$}\hfill\\
   Suppose next that $p\mid (A-1)$, so that $p\ge 3$. Then
   \[{T}(x)\equiv (x+1)^2(x^2-x+1) \pmod{p}.\] It follows that
    \[\overline{T}(x)=\left\{
  \begin{array}{cl}
  (x+1)^4 & \mbox{if $p=3$}\\[.5em]
  (x+1)^2(x^2-x+1)& \mbox{if $p\equiv 2 \pmod{3}$}\\[.5em]
  (x+1)^2(x-r_1)(x-r_2)& \mbox{if $p\equiv 1 \pmod{3}$,}
  \end{array}\right.\]
  where $r_i\in \Z$, such that 
  \[r_1\equiv (1+\sqrt{-3})/2 \pmod{p} \quad \mbox{and} \quad r_2\equiv (1-\sqrt{-3})/2 \pmod{p}.\] Thus, in Theorem \ref{Thm:Dedekind}, we can let
  \[\begin{array}{cl}
   h_1(x)=(x+1) \quad \mbox{and}\quad h_2(x)=(x+1)^3 & \mbox{if $p=3$}\\[.5em]
   h_1(x)=(x+1)(x^2-x+1) \quad \mbox{and}\quad h_2(x)=x+1 & \mbox{if $p\equiv 2 \pmod{3}$}\\[.5em]
   h_1(x)=(x+1)(x-r_1)(x-r_2) \quad \mbox{and}\quad h_2(x)=x+1 & \mbox{if $p\equiv 1 \pmod{3}$.}
  \end{array}\]
  Hence, we see once again that to determine $\gcd(\overline{F},\overline{h_1},\overline{h_2})$, we only need to calculate $\overline{F}(-1)$. Straightforward calculations reveal that
  \[F(-1)=\dfrac{2(A-1)}{p}\]
  in all cases. It follows that
  \[\gcd(\overline{F},\overline{h_1},\overline{h_2})=1 \ \iff \ p\nmid [\Z_K:\Z[\theta]] \ \iff \ p^2\nmid (A-1).\] Consequently, another necessary condition for the monogencicity of $f(x)$ is that $A-1$ is squarefree. Assume then that $A-1$ is squarefree, and proceed.

  \subsection*{The case $p\mid (A+1)$}\hfill\\
  The case when $p\mid (A+1)$ is similar to the case when $p\mid (A-1)$, and we omit the details. In this situation, we have that $p\ge 3$ since $2\mid A$, and we need to calculate
  \[F(1)=-\dfrac{2(A+1)}{p},\]
  which implies that $A+1$ must be squarefree if $f(x)$ is monogenic. Assume that $A+1$ is squarefree, and proceed.

\subsection*{The case $p\mid (A^2+8)$}\hfill\\
  Finally, suppose that $p\mid (A^2+8)$, where $p\ge 3$. Since $2\mid A$, we have that $p\mid ((A/2)^2+2)$. Then
  \begin{align}\label{Eq:Tbar}
  \nonumber T(x)&\equiv x^4+Ax^3+((A/2)^2+2)x^2+Ax+1 \pmod{p}\\
  &\equiv (x^2+(A/2)x+1)^2 \pmod{p},
  \end{align} Let 
  \begin{equation}\label{Eq:gamma}
  \gamma(x):=x^2+(A/2)x+1.
  \end{equation}
  There are several subcases to consider.

  \subsection*{The subcase $p\mid (A^2+8)$ and $p\mid (A^2-16)$}\hfill\\
  If $p\mid (A^2-16)$, then $\gamma(x)$ has a double root. Since $A^2\equiv -8 \pmod{p}$, we have that $p\mid 24$, so that $p=3$.
  If $3\mid (A-4)$, then $A\equiv 1 \pmod{3}$ and $\overline{T}(x)=(x+1)^4$. Easy computations then show that
  \[F(-1)=\dfrac{2(A-1)}{3}\not \equiv 0 \pmod{3},\] since $A-1$ is squarefree.
  If $3\mid (A+4)$, then $\overline{T}(x)=(x-1)^4$, and
  \[F(1)=\dfrac{-2(A+1)}{3}\not \equiv 0 \pmod{3},\] since $A+1$ is squarefree.

\subsection*{The subcase $p\mid (A^2+8)$, $p\nmid (A^2-16)$ and $A^2-16$ is a square in $\F_p$}\hfill\\
   If $p\nmid (A^2-16)$, but $A^2-16$ is a square in $\F_p$, then $\gamma(x)$ has two distinct roots in $\F_p$, namely,
  \[\dfrac{-A\pm \sqrt{A^2-16}}{4}\equiv \dfrac{-(A/2)\pm \sqrt{-6}}{2} \pmod{p}.\] Thus,
  \[\overline{T}(x)=\left(x-\dfrac{-A/2+\sqrt{-6}}{2}\right)^2\left(x-\dfrac{-A/2-\sqrt{-6}}{2}\right)^2.\] Therefore, we can let
  \[h_1(x)=h_2(x)=\left(x-\dfrac{p+1}{2}\left(-A/2+s\right)\right)\left(x-\dfrac{p+1}{2}\left(-A/2-s\right)\right)\] in Theorem \ref{Thm:Dedekind}, where $s\in \Z$ is such that $s^2\equiv -6\pmod{p}$. Then, using Maple, we calculate
   \[\overline{F}(x)=\overline{A}x^3+\overline{3\left(\dfrac{A^2+8}{8p}\right)}x^2\\
  +\overline{\left(3A+\dfrac{A(A^2+8)}{16p}\right)}x+\overline{\left(\dfrac{A^2+8}{8p}+4\right)}.\]
  Next, with $r=(p+1)\left(-A/2\pm s\right)/2$, we calculate
\begin{equation}\label{Eq:F(r)}
\overline{F}(r)=\dfrac{(-8\mp A\sqrt{-6})((A/2)^2+2)}{4p}.
\end{equation} Observe that
\begin{align*}
  -8\mp A\sqrt{-6}\equiv 0 \Mod{p}&\Longrightarrow 64\equiv -6A^2\Mod{p}\\
  &\Longrightarrow 64\equiv 48 \Mod{p}\\
  &\Longrightarrow 16\equiv 0 \Mod{p},
\end{align*} which yields the contradiction that $p=2$. Hence, we conclude from \eqref{Eq:F(r)} that
\begin{equation}\label{Eq:gcd}
\gcd(\overline{F},\overline{h_1},\overline{h_2})=1 \quad \mbox{if and only if}\quad p^2\nmid ((A/2)^2+2),
\end{equation}
which implies that a necessary condition for the monogenicity of $f(x)$ is that $(A/2)^2+2$ is squarefree.

\subsection*{The subcase $p\mid (A^2+8)$, $p\nmid (A^2-16)$ and $A^2-16$ is not a square in $\F_p$}\hfill\\
Finally, the last situation to consider is that $\gamma(x)$ in \eqref{Eq:gamma} is irreducible over $\F_p$. Then
$\overline{T}(x)$ is given in \eqref{Eq:Tbar}, and we can let
\[h_1(x)=h_2(x)=\gamma(x).\] Thus, straightforward computations show that
\[\overline{F}(x)=\left(\overline{\dfrac{(A/2)^2+2}{p}}\right)x^2,\] so that \eqref{Eq:gcd} follows in this situation as well, which completes the proof of item \eqref{M:I10}.

\subsection*{The proof of item \eqref{M:I11}}\hfill\\
Next, for item \eqref{M:I11}, because $2\mid A$, we can let $A=2k$ for some integer $k$. Then, we can rewrite \eqref{Eq:Delf} as
\[\Delta(f)=-16(2k-1)(2k+1)(k^2+2)^2.\] By \cite{BB} (see also \cite[Corollary 2.7]{JW}), there exist infinitely many integers $k$ such that

\[(2k-1)(2k+1)(k^2+2)=(A-1)(A+1)((A/2)^2+2)\] is squarefree. Then, for each such value of $k$, we have that $A-1$, $A+1$ and $(A/2)^2+2$ are all squarefree, which proves that the set $\FF$ is infinite.

\subsection*{The proof of item \eqref{M:I12}}\hfill\\
For item \eqref{M:I12}, let $f(x)\in \FF$. Thus, if either
\[W_1W_2=4(1-A^2) \quad \mbox{or} \quad W=4(1-A^2)(A^2+8)\]
is a square, then $1-A^2\ge 0$, which implies the contradiction $A\in \{0,\pm 1\}$. It follows from Lemma \ref{Lem:AP} that $\Gal(f)\simeq D_4$.

\subsection*{The proof of item \eqref{M:I13}}\hfill\\
For item \eqref{M:I13}, we let
\[f_i(x):=x^4+A_ix^3+A_ix+1\in \FF,\] for $i\in \{1,2\}$. Observe that if $A_2=-A_1$ and $f_1(\theta)=0$, then $f_2(-\theta)=0$, so that $f_1(x)$ is equivalent to $f_2(x)$. Hence, we can focus on the set $\FF^{+}$. Assume then that $A_i\ge 2$ with $A_1\ne A_2$. Let $K_i=\Q(\theta_i)$, where $f_i(\theta_i)=0$. Suppose that $K_1=K_2$. Thus, since each $f_i(x)$ is monogenic, it follows from Theorem \ref{Thm:mono}, that
\[\Delta(f_1)=\Delta(K_1)=\Delta(K_2)=\Delta(f_2),\]
which implies that
\begin{equation}\label{Eq:DiscEquation}
(A_1-1)(A_1+1)(A_1^2+8)^2=(A_2-1)(A_2+1)(A_2^2+8)^2
\end{equation}
from \eqref{Eq:Delf}. Using Maple to solve \eqref{Eq:DiscEquation}, we get several solutions. All solutions with $A_1\ne \pm A_2$ require that
\[2A_1^2+A_2^2+15=\pm \sqrt{-3(A_2-1)(A_2+1)(A_2^2+11)},\] which is impossible since $A_2\ge 2$, and the proof of item \eqref{M:I1} is complete.

\subsection*{The proof of item \eqref{M:I2}}\hfill\\
Since $W_3$ is not a square, we have that $f(x)$ is irreducible over $\Q$ by item \eqref{M1:I4} of Theorem \ref{Thm:Main1}.
Recall $\Delta(f)$ from \eqref{Eq:Delf}.

As in the proof of item \eqref{M:I1}, we use Theorem \ref{Thm:Dedekind} where $T(x):=f(x)$, together with the prime divisors $p$ of $\Delta(f)$, to determine necessary and sufficient conditions for the monogenicity of $T(x)$. Let $K=\Q(\theta)$ with ring of integers $\Z_K$, where $T(\theta)=0$.

\subsection*{The special case of $p=2$}\hfill\\
Recall $\Delta(f)$ from \eqref{Eq:Delf}. Suppose that $2\mid \Delta(f)$. Then we have the following possibilities.

\subsection*{The subcase $2\mid W_1$ and $2\mid W_2$}\hfill\\
Note that $2\mid W_1$ if and only if $2\mid W_2$.
Since $B\equiv 2A-2\equiv 0 \pmod{2}$, we have
\[T(x)\equiv (x+1)^2(x^2+Ax+1) \pmod{2}.\]
Hence,
\[\overline{T}(x)=\left\{\begin{array}{cl}
 (x+1)^4 & \mbox{if $2\mid A$}\\[.5em]
 (x+1)^2(x^2+x+1) & \mbox{if $2\nmid A$.}
\end{array}\right.\] Thus, we can let
\begin{align}\label{Eq:h1h2p=2}
\begin{split}
h_1(x)&=x+1 \quad \mbox{and} \quad h_2(x)=(x+1)^3 \quad  \mbox{if $2\mid A$, and}\\
h_1(x)&=(x+1)(x^2+x+1) \quad \mbox{and} \quad h_2(x)=x+1 \quad \mbox{if $2\nmid A$}.
\end{split}
\end{align} Hence,
\[\overline{F}(x)=\left\{\begin{array}{cl}
  x\left(\overline{\left(\dfrac{A}{2}\right)}x^2+\overline{\left(\dfrac{B+2}{2}\right)}x+ \overline{\left(\dfrac{A}{2}\right)}\right) & \mbox{if $2\mid A$}\\[1.5em]
  x\left(\overline{\left(\dfrac{A+1}{2}\right)}x^2+\overline{\left(\dfrac{B}{2}\right)}x+ \overline{\left(\dfrac{A+1}{2}\right)}\right) & \mbox{if $2\nmid A$.}
 \end{array}\right.\] Note from \eqref{Eq:h1h2p=2} that we only need to calculate $\overline{F}(-1)$ in any case to determine if $\gcd(\overline{F},\overline{h_1},\overline{h_2})=1$. Thus,
 \[\overline{F}(-1)=\left\{\begin{array}{cl}
   0 & \mbox{if $2\mid A$ and $2\mid \mid B \ $ or $ \ 2\nmid A$ and $2^2\mid B$}\\[.5em]
   1 & \mbox{if $2\mid A$ and $2^2\mid B \ $ or $ \ 2\nmid A$ and $2\mid \mid B$.}
 \end{array}\right.\] Consequently,
 \[2\nmid [\Z_K:\Z[\theta]] \ \iff \ (A \mmod{4}, \ B \mmod{4})\in \{(0,0),(2,0),(1,2),(3,2)\}.\]

\subsection*{The subcase $2\nmid W_1W_2$ and $2\mid W_3$}\hfill\\
Note that $2\nmid B$ since $2\nmid W_1W_2$, and $2\mid A$ since $2\mid W_3$. Then
\[\overline{T}(x)=(x^2+x+1)^2 \pmod{2}.\]
Thus, with
\[h_1(x)=h_2(x)=x^2+x+1,\] in Theorem \ref{Thm:Dedekind}, we get that
\[\overline{F}(x)=x\left(\overline{\left(\dfrac{2-A}{2}\right)}x^2+\overline{\left(\dfrac{3-B}{2}\right)}x+\overline{\left(\dfrac{2-A}{2}\right)}\right).\] It follows that $\gcd(\overline{F},\overline{h_1},\overline{h_2})\ne 1$ if and only if
\[\left(\overline{\left(\dfrac{2-A}{2}\right)},\overline{\left(\dfrac{3-B}{2}\right)}\right)\in \{(0,0),(1,1)\},\]
 if and only if
 \[(A \mmod{4}, \ B \mmod{4})\in \{(2,3),(0,1)\}.\] 
 Consequently,
 \[2\nmid [\Z_K:\Z[\theta]] \ \iff \ (A \mmod{4}, \ B \mmod{4})\in \{(0,3),(2,1)\}.\]

 \subsection*{The case $p\mid W_1$ with $p\ge 3$}\hfill\\
 If $p\mid W_1$ with $p\ge 3$, then
 \[T(x)\equiv (x+1)^2\gamma(x) \pmod{p},\]
 where
\begin{equation*}\label{Eq:gam p not 2}
\gamma(x)=x^2+(A-2)x+1.
\end{equation*} There are then the following three subcases to consider for $\gamma(x)$. 

\subsection*{The subcase that $\gamma(x)$ has a double root in $\F_p$}\hfill\\
  If $\gamma(x)$ has a double root in $\F_p$, then $\Delta(\gamma)=A(A-4)\equiv 0 \pmod{p}$. Thus, there are two possibilities to consider here, namely, $p\mid A$ and $p\mid (A-4)$.

If $p\mid A$, then note that $p\mid W_2$. Also,
 \[\gamma(x)\equiv (x-1)^2\pmod{p}\] so that
 \[\overline{T}(x)=(x+1)^2(x-1)^2.\] With
 \[h_1(x)=h_2(x)=(x+1)(x-1)\] in Theorem \ref{Thm:Dedekind}, we get
 \[\overline{F}(x)=-x\left(\overline{\left(\dfrac{A}{p}\right)}x^2+\overline{\left(\dfrac{B+2}{p}\right)}x+\overline{\left(\dfrac{A}{p}\right)}\right).\]
 Then
 \begin{align*}
   \overline{F}(-1)&=-\overline{\left(\dfrac{W_1}{p}\right)}=0 \quad \mbox{if and only if} \quad p^2\mid W_1\\
     \mbox{and} \quad \overline{F}(1)&=-\overline{\left(\dfrac{W_2}{p}\right)}=0 \quad \mbox{if and only if} \quad p^2\mid W_2.
 \end{align*}

 Similarly, if $p\mid (A-4)$, then $\overline{T}(x)=(x+1)^4$ and
 \[\overline{F}(x)=-x\left(\overline{\left(\dfrac{A-4}{p}\right)}x^2+\overline{\left(\dfrac{B-6}{p}\right)}x+\overline{\left(\dfrac{A-4}{p}\right)}\right),\]
 so that
 \[\overline{F}(-1)=-\overline{\left(\dfrac{W_1}{p}\right)}=0 \quad \mbox{if and only if} \quad p^2\mid W_1.\] Thus,
 necessary conditions, in this case, so that $f(x)$ is monogenic are that both $W_1$ and $W_2$ are squarefree.

\subsection*{The subcase that $\gamma(x)$ has two distinct roots in $\F_p$}\hfill\\
Suppose next that $p\nmid A(A-4)$, but that $A(A-4)$ is a square in $\F_p$, so that $\gamma(x)$ has two distinct roots in $\F_p$. Then
 \[\gamma(x)\equiv (x+r)(x+s) \pmod{p},\] where $r,s\in \Z$ such that $r+s\equiv A-2 \pmod{p}$ and $rs\equiv 1 \pmod{p}$. Thus, we can let
 \[h_1(x)=(x+1)(x+r)(x+s) \quad \mbox{and} \quad h_2(x)=x+1\]
 in Theorem \ref{Thm:Dedekind}. Observe then that we only need to calculate $\overline{F}(-1)$ to determine whether $\gcd(\overline{F},\overline{h_1},\overline{h_2})=1$. That is,
 \[\overline{F}(-1)\ne 0 \quad \mbox{if and only if}\quad \gcd(\overline{F},\overline{h_1},\overline{h_2})=1 \quad \mbox{if and only if} \quad p\nmid [\Z_K:\Z[\theta]].\] Using Maple, we get that
 \begin{multline*}
 \overline{F}(x)=\overline{\left(\dfrac{r+s-(A-2)}{p}\right)}x^3+\overline{\left(\dfrac{2(r+s)+rs-B+1}{p}\right)}x^2\\
 +\overline{\left(\dfrac{r+s+2rs-A}{p}\right)}x+\overline{\left(\dfrac{rs-1}{p}\right)},
 \end{multline*}
 and therefore,
 \[\overline{F}(-1)=-\overline{\left(\dfrac{W_1}{p}\right)}=0 \quad \mbox{if and only if} \quad p^2\mid W_1.\]
 Thus, a necessary condition for the monogenicity of $f(x)$ in this case is that $W_1$ is squarefree.

\subsection*{The subcase that $\gamma(x)$ is irreducible over $\F_p$}\hfill\\
 In this situation, omitting the details, we get that
 \[\overline{F}(x)=-\overline{\left(\dfrac{W_1}{p}\right)}x^2.\] Hence,
 \[\overline{F}(-1)=-\overline{\left(\dfrac{W_1}{p}\right)}=0 \quad \mbox{if and only if} \quad p^2\mid W_1.\]

 \subsection*{The case $p\mid W_2$ with $p\ge 3$}\hfill\\
  The case when $p\mid W_2$ is similar to $p\mid W_1$, with some minor differences. In this particular situation we get that
 \[T(x)\equiv (x-1)^2(x^2+(A+2)x+1)\pmod{p},\] and we end up checking only $\overline{F}(1)$.
  The conditions we get for the monogenicity of $f(x)$ simplify interchange $W_1$ and $W_2$ from the case when $p\mid W_1$, and so we omit the details.

 \subsection*{The case $p\mid W_3$ with $p\ge 3$}\hfill\\
 If $p\mid W_3$, then $B\equiv (A^2+8)/4 \pmod{p}$. Hence,
 \[T(x)\equiv (x^2+(A/2)x+1)^2 \pmod{p}.\]
 We let $\gamma(x):=x^2+(A/2)x+1$.
 Since the analysis is similar to the previous situations, we just give a summary of the four cases for $\gamma(x)$.
 \subsection*{$\gamma(x)$ has a double root in $\F_p$; $p\mid (A-4)$}
 \begin{itemize}
 \setlength\itemsep{.5em}
   \item $\overline{T}(x)=(x+1)^4$;
   \item $\overline{F}(x)=-\overline{\left(\dfrac{A-4}{p}\right)}x^3-\overline{\left(\dfrac{B-6}{p}\right)}x^2-\overline{\left(\dfrac{A-4}{p}\right)}x$;
   \item $\overline{F}(-1)=-\overline{\left(\dfrac{W_1}{p}\right)}$.
 \end{itemize}
Conclusion:
\[p\nmid [\Z_K:\Z[\theta]] \ \iff \ p^2\nmid W_1.\]
 \subsection*{$\gamma(x)$ has a double root in $\F_p$; $p\mid (A+4)$}
 \begin{itemize}
 \setlength\itemsep{.5em}
   \item $\overline{T}(x)=(x-1)^4$;
   \item $\overline{F}(x)=-\overline{\left(\dfrac{A+4}{p}\right)}x^3-\overline{\left(\dfrac{B-6}{p}\right)}x^2-\overline{\left(\dfrac{A+4}{p}\right)}x$;
   \item $\overline{F}(1)=-\overline{\left(\dfrac{W_2}{p}\right)}$.
 \end{itemize}
Conclusion:
\[p\nmid [\Z_K:\Z[\theta]] \ \iff \ p^2\nmid W_2.\]
 \subsection*{$\gamma(x)$ has two distinct roots in $\F_p$}
 \begin{itemize}
 \setlength\itemsep{.5em}
   \item $\overline{T}(x)=(x+r)^2(x+s)^2$,\\
    where $r,s\in Z$ with $r+s\equiv A/2 \pmod{p}$ and $rs\equiv 1 \pmod{p}$;
   \item $\overline{F}(x)=\overline{\left(\dfrac{2(r+s)-A}{p}\right)}x^3+\overline{\left(\dfrac{(r+s)^2+2rs-B}{p}\right)}x^2$\\
       $\hspace*{2in} +\overline{\left(\dfrac{2rs(r+s)-A}{p}\right)}x+\overline{\left(\dfrac{r^2s^2-1}{p}\right)}$;
   \item $\overline{F}(-r)=\overline{F}(-s)=\overline{\left(\dfrac{W_3}{4p}\right)}$.
 \end{itemize}
Conclusion:
\[p\nmid [\Z_K:\Z[\theta]] \ \iff \ p^2\nmid W_3.\]
 \subsection*{$\gamma(x)$ is irreducible in $\F_p[x]$}
 \begin{itemize}
 \setlength\itemsep{.5em}
   \item $\overline{T}(x)=\left(x^2+\left(\dfrac{p+1}{2}\right)x+1\right)^2$;
   \item $\overline{F}(x)=\overline{A}x^3+\overline{\left(\dfrac{W_3}{p}+\dfrac{A^2}{2}\right)}x^2+\overline{A}x$\\
   $=\overline{A}x\left(x^2+\overline{\left(\dfrac{A}{2}\right)}x+1\right) \ \iff \ p^2\mid W_3$.
      \end{itemize}
Conclusion:
\[p\nmid [\Z_K:\Z[\theta]] \ \iff \ p^2\nmid W_3.\]
Combining all of these results establishes item \ref{M:I21} under item \eqref{M:I2}.

\subsection{The proof of item \ref{M:I22} under item \eqref{M:I2}}\label{SS:MI22}\hfill\\
It follows from item \eqref{M1:I1}, and by the definition of $\FF_i$, that the sets $\FF_i$ are mutually exclusive and their union contains all monogenic quartic polynomials of these forms. Note that if $f^{+}(\theta)=0$, where
\[f^{+}(x)=x^4+Ax^3+Bx^2+Ax+1 \in \FF_i,\] then
\[f^{-}(x)=x^4-Ax^3+Bx^2-Ax+1\in \FF_i\] and  $f^{-}(-\theta)=0$. Hence, $f^{+}(x)$ and $f^{-}(x)$ are equivalent, and no $\FF_i$ consists solely of distinct polynomials. Nevertheless, for each $i\in \{1,2,3,4\}$, we are able to construct an infinite subset of $\FF_i$, such that all polynomials in the subset are distinct; and for $\FF_5$, we can determine completely a maximal subset consisting solely of distinct polynomials.

\subsection*{The proof that $\FF_1$ contains infinitely many distinct polynomials}\hfill\\
Suppose that $A\ge 6$. Since $2\mid A$, we can write $A=2k$ for some $k\in \Z$ with $k\ge 3$. Let $B=4$.
By \cite{BB,JW}, there exist infinitely many positive integers $k$ such that
$(3-2k)(3+2k)(k^2-2)$ is squarefree, which implies that $3-2k$, $3+2k$ and $k^2-2$ are all squarefree. Note then that $6-4k$ and $6+4k$ are also squarefree. Consequently, for such a value of $k$, the polynomial
\[f_k(x)=x^4+2kx^3+4x^2+2kx+1\] has
\[W_1=6-4k,\quad W_2=6+4k \quad \mbox{and} \quad W_3=4k^2-8,\] where $W_1$, $W_2$ and $W_3/4$ are squarefree. Note also that $W_3$ is not a square since $k^2-2\ne 1$, so that $f(x)$ is irreducible over $\Q$ by item \eqref{M1:I4} of Theorem \ref{Thm:Main1}. Thus, $f_k(x)\in \FF_1$. Let $\FF_{1,k}$ be the subset of $\FF_1$ defined by
\[\FF_{1,k}:=\{f_k(x): (3-2k)(3+2k)(k^2-2) \ \mbox{is squarefree}\}.\] Suppose that $f_{k_1}(x), f_{k_2}(x)\in \FF_{1,k}$, with $k_1>k_2\ge 3$, are equivalent. Since $f_{k_i}(x)$ is monogenic, we have from Theorem \ref{Thm:mono} that $\Delta(f_{k_1})=\Delta(f_{k_2})$, which implies that
\[4(3-2k_1)(3+2k_1)(4k_1^2-8)^2=4(3-2k_2)(3+2k_2)(4k_2^2-8)^2.\]
Solving this equation using Maple we get that $k_1$ must be a root of the polynomial
\[H(z)=4z^4+(4k_2^2-25)z^2+4k_2^4-25k_2^2+52.\]
Recall that $k_2\ge 3$, so that $4k_2^2-25>0$. Hence, it follows that $H(z)$ has no real roots, and thus, the polynomials in the infinite set $\FF_{1,k}$ are all distinct.
\subsection*{The proof that $\FF_2$ contains infinitely many distinct polynomials}\hfill\\
Since the proof is similar to the proof for $\FF_1$, we omit most of the details. In this case, we let $A=2k$ for some $k\in \Z$ with $k\ge 2$, and we let $B=3$. Let
\[f_k(x)=x^4+2kx^3+3x^2+2kx+1,\] and let $\FF_{2,k}$ be the subset of $\FF_2$ defined by
\[\FF_{2,k}:=\{f_k(x): (5-4k)(4k+5)(k^2-1) \ \mbox{is squarefree}\}.\] Since there exist infinitely many positive integers $k$ such that $(5-4k)(4k+5)(k^2-1)$ is squarefree, we have that $\FF_{2,k}$ is infinite. Then, the same techniques used in the proof of $\FF_1$ show that the polynomials in $\FF_{2,k}$ are all distinct.

\subsection*{The proof that $\FF_3$ contains infinitely many distinct polynomials}\hfill\\
Suppose that $A\ge 11$ and $B=10$. Since $2\nmid A$, we let $A=2k+1$ for some integer $k\ge 5$. By \cite{BB,JW}, there exist infinitely many positive integers $k$ such that
\[(5-2k)(7+2k)(4k^2+4k-31)\] is squarefree, which implies that
\[W_1=10-4k,\quad W_2=14+4k \quad \mbox{and} \quad W_3=4k^2+4k-31\] are all squarefree. Note that if $W_3$ is a square, then $W_3=1$ which yields the contradiction that $k=(-1\pm \sqrt{33})/2$. Thus, $W_3$ is not a square. Then, for
 such a value of $k$, we have
\[f_k(x):=x^4+(2k+1)x^3+6x^2+(2k+1)x+1\in \FF_3.\] Let $\FF_{3,k}$ be the subset of $\FF_3$ defined by
\[\FF_{3,k}:=\{f_k(x): (5-2k)(7+2k)(4k^2+4k-31) \ \mbox{is squarefree}\}.\] Suppose that $f_{k_1}(x), f_{k_2}(x)\in \FF_{3,k}$, with $k_1>k_2\ge 5$, are equivalent. Since $f_{k_i}(x)$ is monogenic, we have from Theorem \ref{Thm:mono} that $\Delta(f_{k_1})=\Delta(f_{k_2})$, which implies that
\[-4(2k_1+7)(2k_1-5)(4k_1^2+4k_1-31)^2=-4(2k_2+7)(2k_2-5)(4k_2^2+4k_2-31)^2.\]
Solving this equation using Maple we get that $k_1$ must be a root of the polynomial
\[H(z)=16z^4+32z^3+(C+16)z^2+Cz+(k_2^2+k_2)C+3131,\]
where
\[C:=16k_2^2+16k_2-388\ge 92,\]
since $k_2\ge 5$. Then
\[H'(z)=(2z+1)(32z^2+32z+C),\] so that
$H(z)$ has a unique minimum at $z=-1/2$. Since
\[H(-1/2)=(4k_2^2+4k_2-1)C/4\ge 2737,\]
it follows that $H(z)$ has no real roots. Therefore, we conclude that the polynomials in the infinite set $\FF_{3,k}$ are all distinct.

\subsection*{The proof that $\FF_4$ contains infinitely many distinct polynomials}\hfill\\
Suppose that $A\ge 3$ and $B=1$. Since $2\nmid A$, we let $A=2k+1$ for some integer $k\ge 1$. Then,
\[W_1=1-4k,\ W_2=4k+5 \ \mbox{and} \ W_3=4k^2+4k+5.\]
By \cite{BB,JW}, there exist infinitely many positive integers $k$ such that $W$ is squarefree. Hence, for such a value of $k$, we have that each of $W_1$, $W_2$ and $W_3$ is squarefree. Thus, if $W_3$ is a square, it must be that $W_3=1$, which yields the contradiction that $k=(-1\pm \sqrt{-3})/2$. Similarly, if $W$ is a square, then $W=1$ since $W$ is squarefree. However, solving the equation $W=1$ using Maple shows that the solutions are
\[-1/2\pm \sqrt{-14-2\sqrt{609}}/8 \ \mbox{and} \ -1/2\pm \sqrt{-14+2\sqrt{609}}/8.\]
Thus, $W$ is not a square. Let
\[f_k(x):=x^4+(2k+1)x^3+x^2+(2k+1)x+1\in \FF_4,\] and let $\FF_{4,k}$ be the subset of $\FF_4$ defined by
\[\FF_{4,k}:=\{f_k(x): W_1W_2W_3 \ \mbox{is squarefree}\},\] so that $\FF_{3,k}$ is infinite. Arguing as in previous cases, if any two polynomials  $f_{k_1}(x),f_{k_2}(x)\in \FF_{4,k}$ with $k_1>k_2\ge 1$ are equivalent, calculations indicate that $k_2$ must be a root of the polynomial
\begin{multline*}
H(z)=32z^4+64z^3+(32k_2^2+32k_2+102)z^2\\
+(32k_2^2+32k_2+70)z+32k_2^4+64k_2^3+102k_2^2+70k_2+25.
\end{multline*} It is easy to verify that $H(z)$ has a unique minimum at $z=-1/2$ and that
$H(-1/2)\ge 523/2$. Hence, $H(z)$ has no real roots.

\subsection*{The proof concerning the Galois groups for $f(x)\in \FF_i$}\hfill\\ 
Since $f(x)$ is monogenic for any $f(x)\in \FF_i$ and for any $i\in \{1,2,3,4,5\}$, we have that $W_1$ and $W_2$ are squarefree. If $W_1W_2$ is a square, then $W_1=W_2$, which implies the contradiction that $A=0$. Hence, $W_1W_2$ is not a square. We claim that $W$ is not a square for each $i\in \{1,2,3\}$. We provide details only in the case $i=1$ since the methods are similar for $i\in \{2,3\}$. Since $W$ is not a square for $\FF_4$ by definition, we conclude then that $\Gal(f)\simeq D_4$ from Lemma \ref{Lem:AP} for $f(x)\in \FF_1\cup \FF_2\cup\FF_3\cup \FF_4$.

\subsection*{The proof that $W$ is not a square for $f(x)\in \FF_1$}\hfill\\
If $f(x)\in \FF_1$, then there are the two possibilities
\[(A\mmod{4}, \ B\mmod{4})\in \{(0,0),\ (2,0)\},\] which can be consolidated into the single situation $(A\mmod{2},\ B\mmod{4})=(0,0)$. Thus, we can write $A=2k$ and $B=4m$, for some nonzero integers $k$ and $m$. Assume, by way of contradiction, that $W$ is a square. Let \begin{gather*}
w_1:=W_1/2=2m+1-2k, \quad w_2:=W_2/2=2m+1+2k\\
 \quad \mbox{and}\quad w_3:=W_3/4=k^2-4m+2,
 \end{gather*} so that $w_1$, $w_2$ and $w_3$ are squarefree, and
$W/16=w_1w_2w_3$ is a square. Thus, either $w_1w_2>0$ and $w_3>0$ or $w_1w_2<0$ and $w_3<0$. If $w_1w_2<0$ and $w_3<0$, then
\[(2m+1)^2<4k^2<4(4m-2)=16m-8,\] which yields the contradiction that
\[(2m-3)^2=(2m+1)^2-16m+8<0.\] Hence, $w_1w_2>0$ and $w_3>0$, with $w_1$ and $w_2$ both positive or both negative. Let
\[P:=\gcd(w_1,w_3),\quad Q:=\gcd(w_1,w_2) \quad \mbox{and} \quad R:=\gcd(w_2,w_3).\] Then,
\begin{equation}\label{Eq:PQRsystem}
\abs{w_1}=PQ,\quad \abs{w_2}=QR \quad \mbox{and} \quad w_3=PR,
\end{equation} with $w_1w_2w_3=(PQR)^2$. 

We claim that $P\ne 1$, $Q\ne 1$ and $R\ne 1$. 
Suppose first, by way of contradiction, that $P=1$. Then, from \eqref{Eq:PQRsystem}, we have that $w_1w_3=w_2$. Solving this equation for $m$, we get
\begin{equation}\label{Eq:m}
m=(k^2+4k-1\pm \sqrt{k^4-8k^3+22k^2-56k+9})/8.
\end{equation} Hence, there must be an integral point on the quartic curve
\[y^2=k^4-8k^3+22k^2-56k+9.\] Using the Magma command
\[{\bf IntegralQuarticPoints}([1,-8,22,-56,9]),\] we get the integral points $(k,y)\in \{(-2,\pm 17),\ (0,\pm 3)\}$. We discard the points $(0,\pm 3)$ since $k$ cannot be 0. Thus, $m\in \{3/2,\ -11/4\}$ from \eqref{Eq:m}, contradicting the fact that $m\in \Z$. Hence, $P\ne 1$. The proof that $R\ne 1$ is similar, and we omit the details.
Suppose next, by way of contradiction, that $Q=1$. Then, from \eqref{Eq:PQRsystem}, we get the equation $w_1w_2=w_3$. Solving this equation for $k$ produces the solutions
\[k=\pm \sqrt{20m^2+40m-5}/5,\] which implies the contradiction that $k\not \in \Z$ since $20m^2+40m-5\equiv 3 \pmod{4}$. Thus, the claim is established. Furthermore, since $2\nmid PQR$ and each of $w_1$, $w_2$ and $w_3$ is squarefree, it follows that
\begin{equation}\label{Eq:P Q R}
P,\  Q, \  R \ \mbox{are odd, squarefree, pairwise-coprime and each is larger than 1}.
\end{equation}

 Suppose first that $w_1>0$ and $w_2>0$. 
 Using Maple to solve the system \eqref{Eq:PQRsystem}, we get that
  \begin{equation}\label{Eq:P}
  P^2Q^2-(2Q^2R+16Q+16R)P+(QR-8)^2=0.
  \end{equation} Rearranging \eqref{Eq:P}, we derive the two equivalent equations
  \begin{align*}
    (P^2-2PR+R^2)Q^2-(16P+16R)Q-(PR-4)&=0 \quad \mbox{and}\\
    Q^2R^2-((16+2Q^2)P+16Q)R+(PQ-8)^2&=0.
  \end{align*} We then deduce from these three arrangements, respectively, that
  \[P\mid (QR-8),\quad Q\mid (PR-4) \quad \mbox{and}\quad R\mid (PQ-8).\] Then
  \[PQR\mid (QR-8)(PR-4)(PQ-8),\]
  by \eqref{Eq:P Q R}. Let
  \[Z:=PQ+QR+2PR-8.\]
  Observe that $PQR\mid Z$, since
    \[Z=\dfrac{(QR-8)(PR-4)(PQ-8)-PQR(PQR-8R-8P-4Q)}{32}.\] Let $H:=Z-PQR$. Note that $H\ge 0$ since $PQR\mid Z$, and $Z>0$ from \eqref{Eq:P Q R}. Using the Maple command
    \[{\bf Maximize}(H,\{P\ge a,Q\ge b,R\ge c\}),\]
    with
    \[(a,b,c)\in \{(3,5,7),(3,7,5),(5,3,7),(5,7,3),(7,3,5),(7,5,3)\},\] Maple returns, respectively, the values:
    \[-21, \ -27, \ -7, \ -27, \ -7, \ -21,\] contradicting the fact that $H\ge 0$.
    
    Suppose next that $w_1<0$ and $w_2<0$. Solving the system \eqref{Eq:PQRsystem} using Maple, we get that 
    \begin{equation}\label{Eq:Sol}
    P^2Q^2+(16Q-2Q^2R-16R)P+(QR+8)^2=0.
    \end{equation} Then, solving \eqref{Eq:Sol} for $P$ gives
    \begin{equation}\label{Eq:P}
    P=\dfrac{-8Q+Q^2R+8R\pm 4\sqrt{R(Q^2+4)(R-2Q)}}{Q^2},
    \end{equation} which implies that $R(Q^2+4)(R-2Q)$ is a square. Hence, from \eqref{Eq:P Q R}, we deduce that 
    \begin{equation}\label{Eq:R>2Q}
    R>2Q \quad \mbox{and} \quad R\mid (Q^2+4).
    \end{equation} Similarly, solving \eqref{Eq:Sol} for $R$ gives
    \begin{equation}\label{Eq:P}
    R=\dfrac{PQ^2-8Q+8P\pm 4\sqrt{P(Q^2+4)(P-2Q)}}{Q^2},
    \end{equation} so that 
    \begin{equation}\label{Eq:P>2Q}
    P>2Q \quad \mbox{and} \quad P\mid (Q^2+4).
    \end{equation} Keeping in mind \eqref{Eq:P Q R}, we can combine \eqref{Eq:R>2Q} and \eqref{Eq:P>2Q} to get
    \[PR>4Q^2 \quad \mbox{and}\quad PR\mid (Q^2+4).\] Hence, we have that  
    \[4Q^2<PR\le Q^2+4,\] which implies that $Q^2\le 4/3$, and yields the contradiction that $Q=1$.      
    
    Thus, we conclude that $W$ is not a square. Consequently, from Lemma \ref{Lem:AP}, it follows that $\Gal(f)\simeq D_4$.

\subsection*{The proof for $\FF_5$}\hfill\\
Using methods similar to the previous subsection, it was shown in \cite[pp. 119--121]{JonesOcticsNJM} that the only possible pairs $(A,B)$, such that $W_1$, $W_2$ and $W_3$ are squarefree, and $W$ is a square are
\[(A,B)\in \{(\pm 1,1),(\pm 9, 19), (\pm 11, 31)\}.\] Then,
\[\FF_5=\{f_{A,B}(x): (A,B)\in \{(\pm 1,1),(\pm 9, 19), (\pm 11, 31)\},\] and it is straightforward to confirm that each $f_{A,B}(x)\in \FF_5$ is monogenic with $\Gal(f_{A,B})\simeq C_4$. Recall that $f_{A,B}(x)$ is equivalent to $f_{-A,B}(x)$ from the discussion in subsection \eqref{SS:MI22}. Since
 \[\Delta(f_{1,1})=5^3, \quad \Delta(f_{9,19})=3^213^3 \quad \mbox{and} \quad \Delta(f_{11,31})=5^311^2,\]
the monogenic polynomials $f_{1,1}(x)$, $f_{9,19}(x)$ and $f_{11,31}(x)$ are all distinct.
\end{proof}








\end{document}